\documentclass{amsart}
\usepackage[utf8]{inputenc}

\usepackage{amsmath}
\usepackage{amsfonts}
\usepackage{amssymb}
\usepackage{amsthm}

\usepackage{hyperref}

\usepackage{color}
\usepackage[all]{xy}
\usepackage{enumerate}
\usepackage{mathrsfs}

\usepackage{mycommands}

\newcommand{\B}{\mathcal{B}}

\theoremstyle{plain}
\newtheorem{theorem}{Theorem}[section]
\newtheorem{lemma}[theorem]{Lemma}
\newtheorem*{lemma*}{Lemma}
\newtheorem{proposition}[theorem]{Proposition}
\newtheorem{corollary}[theorem]{Corollary}

\theoremstyle{definition}
\newtheorem{definition}[theorem]{Definition}

\theoremstyle{remark}
\newtheorem{example}[theorem]{Example}
\newtheorem*{example*}{Example}
\newtheorem{remark}[theorem]{Remark}
\newtheorem*{remark*}{Remark}

\title[\texorpdfstring{$p$}{p}-bases and differential operators on varieties...]{\texorpdfstring{$p$}{p}-bases and differential operators on varieties defined over a non-perfect field}

\author{Carlos Abad}

\address{Dpto. Matem\'aticas,  Universidad Aut\'onoma de Madrid and Instituto de Ciencias Ma\-te\-m\'a\-ti\-cas CSIC-UAM-UC3M-UCM, Ciudad Universitaria de Cantoblanco, 28049 Madrid, Spain}

\email{carlos.abad@uam.es}

\thanks{The author wishes to thank his advisor, Orlando Villamayor, for all his suggestions and ongoing support, and Diego Sulca for his comments on section~\ref{sec:regular-varieties}.
The author was partially  supported from the Spanish Ministry of Economy and Competitiveness, through the ``Severo Ochoa'' Programme for Centres of Excellence in R\&D (SEV-2015-0554), and through MTM2015-68524-P (MINECO/FEDER)}


\begin{document}

\begin{abstract}
Let $k$ be a possibly non-perfect field of characteristic $p > 0$. In this work we prove the local existence of absolute $p$-bases for regular algebras of finite type over $k$. Namely, consider a regular variety $Z$ over $k$. Kimura and Niitsuma proved that, for every $\xi \in Z$, the local ring $\OO_{Z,\xi}$ has a $p$-basis over $\OO_{Z,\xi}^p$. Here we show that, for every $\xi \in Z$, there exists an open affine neighborhood of $\xi$, say $\xi \in \Spec(A) \subset Z$, so that $A$ admits a $p$-basis over $A^p$.

This passage from the local ring to an affine neighborhood of $\xi$ has geometrical consequences, some of which will be discussed in the second part of the article. As we will see, given a $p$-basis $\B$ of the algebra $A$ over $A^p$, there is a family of differential operators on $A$ naturally associated to $\B$. These differential operators will enable us to give a Jacobian criterion for regularity for varieties defined over $k$, as well as a method to compute the order of an ideal $I \subset A$.

\end{abstract}

\maketitle

\setcounter{tocdepth}{1}  %
\tableofcontents

\section{Introduction}

Differential methods have been widely used in the study of singularities of algebraic varieties.
While some of these methods are well suited for varieties defined over a perfect field $k$, they tend to fail when $k$ is non-perfect. The aim of this article is to provide some tools and methods based on the use of differential operators which enable us to study the singularities of an embedded variety defined over a non-perfect field.

The difficulties of the non-perfect case already arise when considering a hypersurface embedded in the affine space. Let $f \in k[x_1, \dotsc, x_n]$ be a polynomial defined over a field $k$, and let
\[
	H = \Spec(k[x_1, \dotsc, x_n] / \langle f \rangle)
	\subset \mathbb{A}^n_k
\]%
be the hypersurface defined by $f$ in $\mathbb{A}^n_k$. When the field $k$ is perfect, the singular locus of $H$ is given by
\[
	\Sing(H)
	= \VV \left( \left\langle f,
		\frac{\partial f}{\partial x_1},
		\dotsc, \frac{\partial f}{\partial x_n} \right\rangle \right)
	\subset \mathbb{A}^n_k .
\]%
However, this characterization of the singular locus of $H$ fails in general when $k$ is non-perfect.
For instance, assume that $k = \mathbb{F}_p(v)$, where $\mathbb{F}_p$ represents the prime field of characteristic $p > 0$ and $v$ is a transcendental element over $\mathbb{F}_p$, and consider the polynomial $f = x^p + v y^p \in k[x,y]$. Let $H = \Spec(k[x,y] / \langle f \rangle)$ be the hypersurface defined by $f$ in $\mathbb{A}^2_k$. It can be proved that $H$ is regular everywhere except at the point of the origin. However, $\frac{\partial f}{\partial x} = \frac{\partial f}{\partial y} = 0$, and thus the partial derivatives of $f$ do not help to compute the singular locus of $H$.
As we will see in section~\ref{sec:jacobian}, one way to overcome this difficulty is to consider the partial derivative of $f$ with respect to the transcendental element $v$, say $\frac{\partial f}{\partial v} = y^p$ (see Example~\ref{ex:jacobian-criterion}).

In order to extend some of the existing differential methods to the case of varieties defined over a non-perfect field, we will work with what we call absolute $p$-bases. Let $A$ be a ring of characteristic $p > 0$. An \emph{absolute $p$-basis} of $A$ will be a $p$-basis of $A$ over the subring $A^p \subset A$. That is, an absolute $p$-basis of $A$ will be a subset $\B \subset A$ which has the following properties:
\begin{enumerate}[i)]
\item The elements of $\B$ are \emph{$p$-independent}, i.e., the monomials of the form $b_1^{\alpha_1} \cdot \dotsc \cdot b_r^{\alpha_r}$ with $b_1, \dotsc, b_r \in \B$ and $0 \leq \alpha_i < p$ are linearly independent over $A^p$;
\item The ring $A$, regarded as an $A^p$-algebra, is generated by $\B$. That is, it is required that $A = A^p[\B]$.
\end{enumerate}%
It is known that every field of characteristic $p > 0$ admits an absolute $p$-basis. However, in general, given a ring $C$ of characteristic $p > 0$, there is no an absolute $p$-basis of $C$. Kimura and Niitsuma prove in \cite{KimuraNiitsuma84} that, for a regular variety $Z$ defined over a field $k$ of characteristic $p > 0$, the local ring $\OO_{Z,\xi}$ admits an absolute $p$-basis for every $\xi \in Z$. In this work we prove the local existence of $p$-basis for finitely generated algebras over a non-necessarily perfect field $k$. Namely, given regular variety $Z$ over $k$, we prove that, for every $\xi \in Z$, there exists an open affine neighborhood of $Z$ at $\xi$, say $\Spec(A) \subset Z$, so that $A$ admits an absolute $p$-basis (see Theorem~\ref{thm:variety-admits-p-basis}).

As we will see in section~\ref{sec:differential-bases}, the notion of $p$-basis is closely related to that of differential basis. Given an algebra $A$, we will say that a set $\B \subset A$ is a \emph{differential basis of $A$} if the module of K\"ahler differentials of $A$, say $\Kahler_A$, is a free $A$-module with the set $\{db \in \Kahler_A \mid b \in \B\}$ as a basis. It is not hard to see that an absolute $p$-basis of $A$ is a differential basis of $A$ (the converse fails in general). Thus, given an absolute $p$-basis $A$, say $\B$, one can define a family of derivatives on $A$ which act as partial derivatives with respect to the elements of $\B$. Namely, for each $b \in \B$, one can define a derivative $\frac{\partial}{\partial b} : A \to A$ so that $\frac{\partial}{\partial b}(b') = \delta_{bb'}$ (Kronecker delta) for every $b' \in \B$.
Similarly, in section~\ref{sec:differential-operators} we will show that, when $A$ is a regular domain, there is a family of differential operators of higher order naturally associated to the $p$-basis $\B$ which behave as the differential operators in a polynomial ring (see Proposition~\ref{prop:p-basis:diffs} and Corollary~\ref{crl:p-basis:diff-smooth}). These derivatives and differential operators will be used along sections~\ref{sec:jacobian} and~\ref{sec:differential-saturation-and-order} to give two applications of $p$-bases to the study of singularities on varieties defined over a non-perfect field. This applications are explained in the following lines.

\subsection*{Jacobian criterion for regularity}

Let $k$ be a field of characteristic $p > 0$, and consider the polynomial ring $A = k[x_1, \dotsc, x_n]$. Let $J = \langle f_1, \dotsc, f_s \rangle \subset A$ be an ideal, and set $C = A / J$. Fix a prime $\p \subset C$, and let $P$ denote the preimage of $\p$ in $A$. Assume that $\height(JA_P) = r$, in such a way that $\dim(C_\p) = n-r$. The usual formulation of the Jacobian criterion says that $C_\p$ is smooth over $k$ if and only if the Jacobian matrix
\[
	\begin{pmatrix}
		\frac{\partial f_1}{\partial x_1} & \hdots &
			\frac{\partial f_1}{\partial x_n} \\
		\vdots & \ddots & \vdots\\
		\frac{\partial f_s}{\partial x_1} & \hdots &
			\frac{\partial f_s}{\partial x_n}
	\end{pmatrix}
\]%
has rank $r$ modulo $P$. When $k$ is a perfect field, this is equivalent to say that $C_\p$ is regular. However, the notions of regularity and smoothness differ when $k$ is non-perfect. Namely, when $k$ is non-perfect, it may happen that $C_\p$ is simultaneously smooth over $k$ and non-regular (see \cite{Zariski47}). 
Note that, as opposed to smoothness, which is a relative notion with respect to the field $k$, regularity is an intrinsic property of the local ring $C_\p$.

In section~\ref{sec:jacobian} we formulate a Jacobian criterion for regularity which works over non-perfect fields. Let $k$ be a non-perfect field of characteristic $p > 0$, and let $A = k[x_1, \dotsc, x_n]$, $J = \langle f_1, \dotsc, f_s \rangle$, and $C = A / J$ be as in the previous paragraph. Fix an absolute $p$-basis of $k$, say $\B_0$. Then it is not hard to check that $\B = \B_0 \cup \{x_1, \dotsc, x_n\}$ is an absolute $p$-basis of $A$. Thus, for each $b \in \B_0$, one can define a partial derivative on $A$ associated to $b$ which we denote by $\frac{\partial}{\partial b} : A \to A$. As $\B$ can also be regarded as a differential basis of $A$, one has that $\frac{\partial f_1}{\partial b} = \dotsb = \frac{\partial f_s}{\partial b} = 0$ for almost every element $b \in \B_0$, except for a finite number of them, say $b_1, \dotsc, b_m \in \B_0$. Fix a prime $\p \subset C$, let $P$ denote its preimage in $A$, and set $r = \height(JA_P)$. Under these hypotheses, Zariski's Jacobian criterion \cite[Theorem~11]{Zariski47} says that $C_\p$ is regular if and only if the extended Jacobian matrix
\[
	\begin{pmatrix}
		\frac{\partial f_1}{\partial b_1} & \hdots &
			\frac{\partial f_1}{\partial b_m} &
			\frac{\partial f_1}{\partial x_1} & \hdots &
			\frac{\partial f_1}{\partial x_n} \\
		\vdots & \ddots & \vdots & \vdots & \ddots & \vdots \\
		\frac{\partial f_s}{\partial b_1} & \hdots &
			\frac{\partial f_s}{\partial b_m} &
			\frac{\partial f_s}{\partial x_1} & \hdots &
			\frac{\partial f_s}{\partial x_n}
	\end{pmatrix}
\]%
has rank $r$ modulo $P$.
In Proposition~\ref{prop:jacobian-criterion-p-bases} we will give a more general formulation of Zariski's criterion in which $A$ is not assumed to be a polynomial ring, but simply a regular ring that admits an absolute $p$-basis.

\subsection*{Differential operators and order of ideals}

Consider a regular variety $Z$ over a field $k$. Given a coherent ideal $\II \subset \OO_Z$, the order of $\II$ at a point $\xi \in Z$ is defined by
\[
	\ord_\xi(\II)
	= \max \left \{ n \in \mathbb{N}
		\mid \II_\xi \subset \M_\xi^n \right \},
\]%
where $\M_\xi$ denotes the maximal ideal of $\OO_{Z,\xi}$. For instance, when $Z$ is affine, say $Z = \Spec(A)$, and $\II$ is a principal ideal generated by an element $f \in A$, then $\ord_\xi(\II)$ coincides with the maximum integer $n$ such that $f \in \M_\xi^n$. Fix an integer $N \geq 1$. From the point of view of resolution of singularities, it is interesting to study the set of points of $Z$ where a given ideal $\II \subset \OO_Z$ has at least order $N$, as this set coincides with the singular locus of the idealistic exponent $(\II,N)$ as defined by Hironaka~\cite{Hironaka77}. The purpose of section~\ref{sec:differential-saturation-and-order} is to find a description of this set of points or, more precisely, to find coherent ideal $\mathcal{J}_N \subset \OO_Z$ such that
\begin{equation} \label{eq:intro:order-as-zero-set}
	\{ \xi \in Z \mid \ord_\xi(\II) \geq N \}
	= \VV (\mathcal{J}_N)
	\subset Z .
\end{equation}%
When $Z$ is a regular variety defined over a perfect field $k$, such ideal can be constructed by applying differential operators of order at most $N-1$ over $k$ to the elements of $\II$ (see \cite[\S3]{Hironaka05} and \cite[Ch.~III, Lemma~1.2.7]{GiraudEtudeLocale}). However, this approach fails in general when $k$ is non-perfect for the following reasons. On the one hand, it is not enough to consider $k$-linear differential operators on $Z$. Instead, we shall work with absolute differential operators. That is, with differential operators on $Z$ which are relative to the prime field of $k$, say $\mathbb{F}_p$, rather than those which are relative to $k$. On the other hand, the sheaf of absolute differential operators of order at most $N-1$ may not be finitely generated or quasi-coherent over $Z$. Still, given a regular variety $Z$ over a non-perfect field $k$, a coherent ideal $\II \subset \OO_Z$, and an integer $N \geq 1$, we will show that it is possible to construct a coherent ideal $\mathcal{J}_N \subset \OO_Z$ as that in \eqref{eq:intro:order-as-zero-set} by applying absolute differential operators to the elements of $\II$ (see Proposition~\ref{prop:coherent-diff-I} and Theorem~\ref{thm:order-ideal-sheaf}).

\section{Definition and first properties of \texorpdfstring{$p$}{p}-bases}

\label{sec:p-bases}

In this section we review the notion of $p$-basis, as well as some examples and some of the main properties of $p$-bases.

\begin{definition}
Let $k$ be a ring containing $\mathbb{F}_p$, let $A$ be an arbitrary $k$-algebra, and consider a (possibly infinite) subset $\B \subset A$. A \emph{$p$-monomial in $\B$} is a monomial of the form $b_1^{\alpha_1} \cdot\dotsc\cdot b_r^{\alpha_r}$, with $b_1, \dotsc, b_r \in \B$, $r \in \mathbb{N}$, and $0 \leq \alpha_i < p$. The subset $\B \subset A$ is said to be \emph{$p$-independent over $k$} if the set of $p$-monomials in $\B$, say
\[
	\left \{ b_1^{\alpha_1} \cdot\dotsc\cdot b_r^{\alpha_r}
	\mid b_i \in \B, \; 0 \leq \alpha_i < p \right \},
\]%
is linearly independent over the subring $A^p[k]$. The subset $\B$ is said to be a \emph{$p$-basis of $A$ over $k$} if it is $p$-independent over $k$ and
\[
	A = A^p[k][\B].
\]%
A $p$-basis of $A$ over the prime field $\mathbb{F}_p$ (or equivalently over $A^p$) is also called an \emph{absolute $p$-basis of $A$}.
\end{definition}

Observe that, if $\B$ is a $p$-basis of $A$ over $k$, then $A$ can be regarded as a free $A^p[k]$-module with the set of $p$-monomials in $\B$ as a basis, and vice-versa. Note also that, given a set of variables, say $\{ t_b\}_{b \in \B}$, there is a natural morphism of $A^p[k]$-algebras from $A^p[k][t_b \mid b \in \B]$ to $A$ mapping each variable $t_b$ to the corresponding element $b$. Thus one can characterize a $p$-basis of $A$ over $k$ as follows.

\begin{lemma}  \label{lm:p-basis-presentation}
A subset $\B \subset A$ is a $p$-basis of $A$ over $k$ if and only if
\[
	A \simeq A^p[k][t_b \mid b \in \B]
		/ \langle t_b^p - b^p \mid b \in \B \rangle ,
\]%
where $\{ t_b\}_{b \in \B}$ represents a set of variables.
\end{lemma}

\begin{proof}
Given an arbitrary set $\B \subset A$, the ideal $\langle t_b^p - b^p \mid b \in \B \rangle$ is clearly contained in kernel of the natural map from $A^p[k][t_b \mid b \in \B]$ to $A$ which maps $t_b$ to $b$. Since
\[
	A^p[k][t_b \mid b \in \B] / \langle t_b^p - b^p \mid b \in \B \rangle
\]%
is a free $A^p[k]$-module with the set of $p$-monomials in $\{ t_b\}_{b \in \B}$, say
\[
	\{t_{b_1}^{\alpha_1} \cdot\dotsc\cdot t_{b_r}^{\alpha_r}
	\mid b_i \in \B, \; 0 \leq \alpha_i < p \},
\]%
as a basis, one readily checks that $\B$ is a $p$-basis of $A$ over $k$ if and only if the natural map
\[ \xymatrix {
	A^p[k][t_b \mid b \in \B] / \langle t_b^p - b^p \mid b \in \B \rangle
	\ar[r] &
	A
} \]%
is an isomorphism.
\end{proof}

\subsection*{On the existence of $p$-bases}
In general, given a ring $k$ and an arbitrary $k$-algebra $R$, the algebra $R$ does not admit a $p$-basis over $k$. For instance, if $k$ is a field and $[k : k^p] = \infty$, then the power series ring $k[[x]]$ does not admit a $p$-basis over $k$ or over $\mathbb{F}_p$ (see \cite[Example~3.4]{KimuraNiitsuma80}). However, there are some cases in which it is known that a ring $R$ does admit a $p$-basis over $k$.

Consider a field extension $K / k$. In this case, one can check that a maximal $p$-independent set over $k$, say $\B \subset K$, is a $p$-basis of $K$ over $k$, and conversely (see \cite[\S26, p.~202]{MatsumuraRingTheory}). %
Thus, by Zorn's lemma, $K$ admits a $p$-basis over $k$. Unfortunately, as the following example shows, this argument does not work in a wider setting.

\begin{example}
Consider the polynomial ring $\mathbb{F}_2[x,y]$, where $x, y$ represent variables.  It can be checked that $\{ x, xy \}$ is a maximal $p$-independent set over $\mathbb{F}_2$, but it is not a $p$-basis of $\mathbb{F}_2[x,y]$ over $\mathbb{F}_2$.
\end{example}

For a polynomial ring $k[x_1, \dotsc, x_n]$, the set of variables $\{x_1, \dotsc, x_n\}$ is always a $p$-basis of $k[x_1, \dotsc, x_n]$ over $k$. This property can be generalized as follows.

\begin{lemma}  \label{lm:p-basis:polynomial-extension}
Let $k$ be a ring over $\mathbb{F}_p$, and let $A_0$ be an arbitrary $k$-algebra which admits a $p$-basis over $k$, say $\B_0$. Consider the polynomial ring $A = A_0[x_1, \dotsc, x_n]$. Then the set $\B = \B_0 \cup \{x_1, \dotsc, x_n\}$ is a $p$-basis of $A$ over $k$.
\end{lemma}

\begin{proof}
Observe that $A^p[k] = A_0^p[k][x_1^p, \dotsc, x_n^p]$. Since $A_0$ is a free $A_0^p[k]$-module with the the set of monomials
\[
	\left \{b_1^{\alpha_1} \cdot\dotsc\cdot b_r^{\alpha_r}
	\mid b_i \in \B_0, \; 0 \leq \alpha_i < p \right \}
\]%
as a basis, it follows that $A$ is a free $A^p[k]$-module with the set of monomials
\[
	\left \{b_1^{\alpha_1} \cdot\dotsc\cdot b_r^{\alpha_r}
		x_1^{\beta_1} \cdot \dotsc \cdot x_n^{\beta_n}
	\mid b_i \in \B, \; 0 \leq \alpha_i < p, \; 0 \leq \beta_j < p \right \}
\]%
as a basis. Hence $\B$ is a $p$-basis of $A$ over $k$.
\end{proof}

\begin{remark}
The reader should be aware that, in general, under the hypotheses of the lemma, the set $\B_0 \cup \{x_1, \dotsc, x_n\}$ is not a $p$-basis of $A_0 [[ x_1, \dotsc, x_n ]]$ over $k$ (see \cite[Example~3.4]{KimuraNiitsuma80}).
\end{remark}

\begin{lemma}[Localization]  \label{lm:p-basis:localization}
Let $k$ be a ring over $\mathbb{F}_p$ and let $A$ be an arbitrary $k$-algebra which has a $p$-basis over $k$, say $\B$. Consider a multiplicative subset $S \subset A$ which does not contain nilpotent elements. Then the image of $\B$ in $S^{-1}A$ is a $p$-basis of $S^{-1}A$ over $k$.
\end{lemma}

\begin{proof}
Set $S^{p} = \{ s^p \mid s \in S \}$, which is a multiplicative subset of $A^p[k]$. Observe that
\[
	S^{-1}A
	= (S^p)^{-1} A
	= A \otimes_{A^p[k]} (S^p)^{-1} (A^p[k]).
\]%
Since $A$ is a free $A^p[k]$-module with the set of monomials of the form $b_1^{\alpha_1} \cdot\dotsc\cdot b_r^{\alpha_r}$ with $b_i \in \B$ and $0 \leq \alpha_i < p$ as basis, the image of these monomials is also a basis of $S^{-1}A$ over $(S^p)^{-1} (A^p[k])$. Hence the image of $\B$ is a $p$-basis of $S^{-1}A$ over $k$.
\end{proof}

\begin{remark}
Let $\B$ be a $p$-basis of an algebra $A$ over a ring $k$. Suppose that a $p$-monomial in $\B$, say $b_1^{\alpha_1} \cdot\dotsc\cdot b_r^{\alpha_r}$, has a zero-divisor in $A$, say $a \cdot b_1^{\alpha_1} \cdot\dotsc\cdot b_r^{\alpha_r} = 0$. In this case,
\[
	a^p \cdot b_1^{\alpha_1} \cdot\dotsc\cdot b_r^{\alpha_r} = 0.
\]%
Then, by the property of $p$-independence of $\B$, we have that $a^p = 0$. Thus we see that all the zero-divisors of the $p$-monomials in $\B$ are nilpotent.
\end{remark}

\section{Differential bases, \texorpdfstring{$p$}{p}-bases, and derivatives}

\label{sec:differential-bases}

Fix an arbitrary ring $k$ containing $\mathbb{F}_p$, and let $A$ be an arbitrary $k$-algebra. The module of K\"ahler differentials of $A$ over $k$ is defined by $\Kahler_{A/k} = I_{A/k} / I_{A/k}^2$, where $I_{A/k}$ represents the kernel of the diagonal morphism $\delta : A \otimes_k A \to A$ given by $\delta(a_1 \otimes a_2) = a_1 a_2$. We shall denote the universal derivation of $A$ over $k$ by $d_{A/k} : A \to \Kahler_{A/k}$. Recall that, regarded as an $A$-module, $\Kahler_{A/k}$ is generated by the elements of the form $d_{A/k}(a)$ with $a \in A$.

\begin{definition}  \label{def:differential-basis}
We will say that a subset $\B \subset A$ is a \emph{differential basis of $A$ over $k$} if $\Kahler_{A/k}$ is a free $A$-module with the set $d_{A/k}(\B) = \{d_{A/k}(b) \mid b \in \B \}$ as a basis. A differential basis of $A$ over the prime field $\mathbb{F}_p$ will also be called an \emph{absolute differential basis} of $A$.
\end{definition}

As it will be shown in the following lines, if $\B$ is a $p$-basis of $A$ over $k$, then $\B$ is a differential basis of $A$ over $k$ (see Proposition~\ref{prop:differential-basis}). Conversely, if $\B$ is a differential basis of $A$ over $k$, then $\B$ is $p$-independent over $k$ (see Lemma~\ref{lm:diff-basis-implies-p-indep}). However, in general, a differential basis of $A$ over $k$ might not be a $p$-basis of $A$, as $A^p[k][\B]$ could be a proper subset of $A$.

\begin{proposition} \label{prop:differential-basis}
Let $A$ be an algebra over a ring $k$ containing $\mathbb{F}_p$. If $\B$ is a $p$-basis of $A$ over $k$, then $\Kahler_{A/k}$ is a free $A$-module with the set $d_{A/k}(\B)$ as basis \textup{(}i.e., $\B$ is a differential basis of $A$ over $k$\textup{)}.
\end{proposition}

\begin{proof}
Set $C_0 = A^p[k]$. Consider the polynomial ring $C = C_0[t_b \mid b \in \B]$  and the ideal $J = \langle t_b^p - b^p \mid b \in \B \rangle \subset C$. Recall that, by Lemma~\ref{lm:p-basis-presentation}, $A = C / J$. Note also that $\Kahler_{A/k} = \Kahler_{A/C_0}$. Then, by
\cite[Theorem~25.2]{MatsumuraRingTheory},
there is an exact sequence of $A$-modules
\[ \xymatrix{
	J / J^2  \ar[rr]^-{\overline{d_{C / C_0}}} &
	&
	\Kahler_{C / C_0} \otimes_{C} A  \ar[r] &
	\Kahler_{A / k}  \ar[r] &
	0.
} \]%
Since $t_b^p \in C^p$ and $b^p \in C_0$, one readily checks that $J / J^2$ is mapped to zero via $\overline{d_{C / C_0}}$. Thus we get a natural isomorphism, say
\[ \xymatrix{
	\Kahler_{C / C_0} \otimes_{C} A  \ar[r]^-{\sim} &
	\Kahler_{A / k},
} \]%
which maps the class of $d_{C / C_0}(t_b)$ to $d_{A/k}(b)$ for each $b \in \B$. Hence $\Kahler_{A/k}$ is a free $A$-module with the set $d_{A/k}(\B) = \{ d_{A/k}(b) \mid b \in \B \}$ as a basis.
\end{proof}

\begin{corollary}  \label{crl:partial-b}
Let $A$ be an algebra over a ring $k$ containing $\mathbb{F}_p$. Assume that $A$ admits a $p$-basis over $k$, say $\B$. Then, for each $b \in \B$, there exists a unique $k$-linear derivative from $A$ to $A$, which we shall denote by $\frac{\partial}{\partial b} : A \to A$, such that $\frac{\partial}{\partial b} (c) = \delta_{bc}$ \textup{(}Kronecker delta\textup{)} for all $c \in \B$.
\end{corollary}

\begin{proof}
$\Der_k(A) \simeq \Hom_A(\Kahler_{A/k}, A)$ (cf. \cite[\S25, p.~192]{MatsumuraRingTheory}).
\end{proof}

\begin{corollary}
Let $\B$ be a $p$-basis of $A$ over $k$ as in the previous corollary. Then, for each $f \in A$, there is a finite number of elements in $\B$, say $b_1, \dotsc, b_m \in \B$, such that $\frac{\partial}{\partial b_i}(f) \neq 0$.
\end{corollary}

\begin{remark}
Let $\B$ be a $p$-basis of $A$ over $k$ as in Proposition~\ref{prop:differential-basis}, and consider a $k$-linear derivative $\delta \in \Der_k(A)$. Since $\Kahler_{A/k}$ is a free $A$-module with the set $\{ d_{A/k}(b) \mid b \in \B \}$ as a basis, one readily checks that
\[
	\delta
	= \sum_{b \in \B} \delta(b) \cdot \frac{\partial}{\partial b}.
\]%
Note that, in principle, this might be an infinite sum. However, given a fixed element $f \in A$, the previous corollary says that there is a finite number of elements $b_1, \dotsc, b_m \in \B$ such that $\frac{\partial}{\partial b_i}(f) \neq 0$. Thus $\delta(f) = \sum_{i = 1}^m \delta(b_i) \cdot \frac{\partial}{\partial b_i}(f),$
which is a finite sum.
\end{remark}

\begin{lemma}  \label{lm:diff-basis-implies-p-indep}
Let $A$ be an algebra over a ring $k$ containing $\mathbb{F}_p$. Assume that a subset $\B$ is a differential basis of $A$ over $k$. Then $\B$ is $p$-independent over $k$.
\end{lemma}

\begin{proof}
We will proceed by contradiction. Assume that there exists a non-trivial relation of $p$-dependence in $\B$, say $F(b_1, \dotsc, b_r) = 0$, where $b_1, \dotsc, b_r \in \B$, and
\[
	F(b_1, \dotsc, b_r)
	= \sum_{0 \leq \alpha_1, \dotsc, \alpha_r < p}
		a_{\alpha_1, \dotsc, \alpha_r}
		b_1^{\alpha_1} \cdot \dotsc \cdot b_r^{\alpha_r}
	\in A^p[k][b_1, \dotsc, b_r].
\]%
Then choose $\alpha_1, \dotsc, \alpha_r$ so that $\alpha_1 + \dotsb + \alpha_r$ is maximum with $a_{\alpha_1, \dotsc, \alpha_r} \neq 0$. Since $\B$ is a differential basis of $A$ over $k$, one can define derivatives $\frac{\partial}{\partial b_1}, \dotsc, \frac{\partial}{\partial b_r}$ from $A$ to $A$ such that $\frac{\partial}{\partial b_i}(b_j) = \delta_{ij}$ (Kronecker delta) for all $i,j$. Note that all these derivatives are $A^p[k]$-linear. Thus one readily checks that
\[
	\frac{1}{\alpha_1 ! \dotsb \alpha_r !} \cdot
	\frac{\partial^{\alpha_1 + \dotsb + \alpha_r}}
		{\partial b_1^{\alpha_1} \dotsb \partial b_r^{\alpha_r}}
	F(b_1, \dotsc, b_r)
	= a_{\alpha_1, \dotsc, \alpha_r} \neq 0.
\]%
Since $F(b_1, \dotsc, b_r) = 0$, this leads to a contradiction.
\end{proof}

\section{Absolute \texorpdfstring{$p$}{p}-bases on regular varieties}

\label{sec:regular-varieties}

\begingroup  %

Consider regular variety $Z$ defined over a possibly non-perfect field $k$ of characteristic $p > 0$. In this section we will prove that, for every $\xi \in Z$, there exists an affine neighborhood of $Z$ at $\xi$, say $\xi \in \Spec(A) \subset Z$, so that $A$ admits an absolute $p$-basis (see Theorem~\ref{thm:variety-admits-p-basis}).

\begin{theorem}[Kimura and Niitsuma \cite{KimuraNiitsuma84}]
\label{thm:KimuraNiitsuma}
Let $R$ be regular local ring of characteristic $p > 0$. Assume that $R$ has an absolute differential basis, say $\B$ \textup{(}that is, $\B$ is a differential basis of $R$ over $\mathbb{F}_p$\textup{)}. Then $\B$ is an absolute $p$-basis of $R$.
\end{theorem}

As we show below, this result can be easily extended to the case of regular domains which are not necessarily local.

\begin{corollary} \label{crl:KimuraNiitsuma:open}
Let $A$ be a regular domain of characteristic $p > 0$. Suppose that $A$ has an absolute differential basis, say $\B$. Then $\B$ is an absolute $p$-basis of $A$.
\end{corollary}

\begin{proof}
The set $\B$ is $p$-independent by Lemma~\ref{lm:diff-basis-implies-p-indep}. Thus we just need to check that $A^p[\B] = A$. To this end, let us regard the algebras $A$ and $A^p[\B]$ as modules over $A^p$. Then, in order to check that $A^p[\B] = A$, it suffices to see that
\begin{equation}  \label{eq:crl-KN:equality}
	A^p[\B] \otimes_{A^p} (A^p)_{\q}
	= A \otimes_{A^p} (A^p)_{\q}
\end{equation}%
for every prime ideal $\q \subset A^p$ (cf. \cite[Proposition~3.9]{AtiyahMacdonaldCommAlg}).

Fix a prime ideal $\q \subset A^p$. Note that the primes of $A$ are in natural correspondence with those of $A^p$. Hence $\q$ must be of the form $\q = Q \cap A^p$ for some prime ideal $Q \subset A$. In order to prove \eqref{eq:crl-KN:equality}, observe that $(A^p \setminus \q) = (A \setminus Q)^p$. Thus one readily checks that $A \otimes_{A^p} (A^p)_{\q} = A_Q$ and $(A^p)_\q = (A_Q)^p$. On the other hand, note that $\B$ is also a differential basis of $A_Q$, since $\Kahler_{A_Q} = \Kahler_A \otimes_A A_Q$ (see \cite[Exercise~25.4]{MatsumuraRingTheory}). %
Hence $(A_Q)^p[\B] = A_Q$ by Theorem~\ref{thm:KimuraNiitsuma}. In this way,
\[
	A^p[\B] \otimes_{A^p} (A^p)_{\q}
	= (A^p)_{\q}[\B]
	= (A_Q)^p[\B]
	= A_Q
	= A \otimes_{A^p} (A^p)_{\q} ,
\]%
which proves \eqref{eq:crl-KN:equality}. 

Since the previous argument works for any $\q \subset A_p$, we conclude that $A^p[\B] = A$, and hence $\B$ is an absolute $p$-basis of $A$.
\end{proof}

\begin{lemma}%
\label{lm:2nd-exact-seq-local}
Let $(R,\M,K)$ be a noetherian local ring containing a field. Then there is a natural short exact sequence
\[ \xymatrix {
	0 \ar[r] &
	\M / \M^2 \ar[r]^-{\overline{d_R}} &
	\Kahler_R \otimes K \ar[r] &
	\Kahler_K \ar[r] &
	0.
} \]%
\end{lemma}

\begin{proof}
Since $K$ is separable over its prime field, the result follows from \cite[Theorem~25.2]{MatsumuraRingTheory}.
\end{proof}

\begin{proposition}  \label{prop:rsp}
Let $(R,\M,K)$ be a regular local ring that has an absolute $p$-basis, say $\B^*$. Then, for any regular system of parameters of $R$, say $z_1, \dotsc, z_d$, there exists a subset $\B_0 \subset \B^*$ so that $\B_0 \cup \{z_1, \dotsc, z_d\}$ is an absolute $p$-basis of $R$.
\end{proposition}

\begin{proof}
Note that, given a regular system of parameters of $R$, say $z_1, \dotsc, z_d$, the set $\{ \overline{z}_1, \dotsc, \overline{z}_d\}$ is a basis of $\M / \M^2$. In addition, by Proposition~\ref{prop:differential-basis}, we know that $\Kahler_R$ is a free $R$-module with $\{d_R(b) \mid b \in \B^*\}$ as a basis, and
\[
	d_R (z_i)
	= \sum_{b \in \B^*}
		\frac{\partial z_i}{\partial b} \cdot d_R(b)
\]%
for $i = 1, \dotsc, d$. Since the induced map $\overline{d_R} : \M / \M^2 \to \Kahler_R \otimes K$ is injective (Lemma~\ref{lm:2nd-exact-seq-local}), there should be elements $b_1, \dotsc, b_d \in \B^*$ such that the Jacobian matrix
\[
	J = \begin{pmatrix}
		\frac{\partial z_1}{\partial b_1} & \hdots &
			\frac{\partial z_1}{\partial b_d} \\
		\vdots & \ddots & \vdots\\
		\frac{\partial z_d}{\partial b_1} & \hdots &
			\frac{\partial z_d}{\partial b_d} \end{pmatrix}
	\in \Mat_{d \times d}(R)
\]%
has rank $d$ modulo $\M$ (i.e., as a matrix over $K = R/\M$). Set $\B_0 = \B^* \setminus \{b_1, \dotsc, b_d\}$. Then we claim that the set $\B = \B_0 \cup \{z_1, \dotsc, z_d\}$ is an absolute $p$-basis of $R$.

In order to prove the claim, we shall show that $\{d_R(b) \mid b \in \B\}$ is a basis of $\Kahler_R$ (recall that, as $\B^*$ is differential basis of $R$, we already know that $\Kahler_R$ is a free $R$-module). By Theorem~\ref{thm:KimuraNiitsuma}, this proves that $\B$ is an absolute $p$-basis of $R$.

Note that the matrix $J$ has non-zero determinant modulo $\M$, as it has rank $d$ modulo $\M$. Hence the determinant of $J$ is a unit in $R$, and therefore $J$ is invertible over $R$. Since
\[
	\{d_R(b) \mid b \in \B^* \}
	= \{d_R(b) \mid b \in \B_0\} \cup \{d_R(b_1), \dotsc, d_R(b_d)\}
\]%
is a basis of $\Kahler_R$, and $J$ is invertible, one readily checks that
\[
	\{d_R(b) \mid b \in \B \}
	= \{d_R(b) \mid b \in \B_0\} \cup \{d_R(z_1), \dotsc, d_R(z_d)\}
\]%
is also a basis of $\Kahler_R$. Hence $\B = \B_0 \cup \{z_1, \dotsc, z_d\}$ is an absolute $p$-basis of $R$.
\end{proof}

\begin{remark}
Fix a regular local ring $(R,\M,K)$, a regular system of parameters $z_1, \dotsc, z_d$, and an absolute $p$-basis of $R$ of the form $\B_0 \cup \{z_1, \dotsc, z_d\}$ as in Proposition~\ref{prop:rsp}. Note that, attending to the short exact sequence of Lemma~\ref{lm:2nd-exact-seq-local}, $\left \{d_K \bigl(\overline{b}\bigr) \mid b \in \B_0 \right \}$ should be basis of $\Kahler_K$, and hence the image of $\B_0$ in $K$ is an absolute $p$-basis of $K$. %
By \cite[Theorem~26.8]{MatsumuraRingTheory}, this implies that the image of $\B_0$ in $K$ is algebraically independent over the prime field $\mathbb{F}_p$. As a consequence, $\mathbb{F}_p[\B_0]$ can be regarded as polynomial ring contained in $R$ and, since $\mathbb{F}_p[\B_0] \cap \M = \{ 0\}$, we may write $\mathbb{F}_p(\B_0) \subset R$. Note also that, since the image of $\B_0$ in $K$ is an absolute $p$-basis of $K$, the field $K$ is formally \'etale over $\mathbb{F}_p(\B_0)$ (cf. \cite[Theorem~26.7]{MatsumuraRingTheory}). %
Hence $\mathbb{F}_p(\B_0)$ can be extended to a unique coefficient field of $\widehat{R}$ (the completion of $R$ with respect to $\M$). This proves that, under the hypotheses of Proposition~\ref{prop:rsp}, the subfield $\mathbb{F}_p(\B_0) \subset R$ is a quasi-coefficient field of $R$.
\end{remark}

\begin{remark}
Let $\B_0 \cup \{z_1, \dotsc, z_d\}$ be an absolute $p$-basis of a regular local ring $(R,\M,K)$ as in Proposition~\ref{prop:rsp}. In the previous remark we showed that $\mathbb{F}_p(\B_0) \subset R$ is a quasi-coefficient field of $R$. Therefore $\mathbb{F}_p(\B_0)$ can be extended to a unique coefficient field of $\widehat{R}$, say $K_0 \subset \widehat{R}$, for which $\B_0$ is an absolute $p$-basis. In addition,
Cohen's structure theorem %
says that $\widehat{R} \simeq K_0 [[z_1, \dotsc, z_d]$. Nevertheless, it should be noted that, in general, $\B_0 \cup \{z_1, \dotsc, z_d\}$ is not a $p$-basis of $\widehat{R}$. Indeed, it can be proved that $\B_0 \cup \{z_1, \dotsc, z_d\}$ is an absolute $p$-basis of $\widehat{R}$ if and only if $\B_0$ is finite (cf. \cite[Example~3.4]{KimuraNiitsuma80}).
\end{remark}

\begin{lemma} \label{lm:p-basis:rsp-open}
Let $A$ be a regular ring of characteristic $p > 0$ that has an absolute $p$-basis, say $\B$. Fix a prime ideal $P \subset A$. Then, for any regular system of parameters $z_1, \dotsc, z_d$ of $A_P$, there exists a subset $\B_0 \subset \B$ and an element $f \in A \setminus P$ so that $z_1, \dotsc, z_d \in A_f$, and $\B_0 \cup \{z_1, \dotsc, z_d\}$ is an absolute $p$-basis of $A_f$.
\end{lemma}

\begin{proof}
Arguing as in the proof of Proposition~\ref{prop:rsp}, one can find a finite collection of elements $b_1, \dotsc, b_d \in A_P$ such that the Jacobian matrix
\[
	J = \begin{pmatrix}
		\frac{\partial z_1}{\partial b_1} & \hdots &
			\frac{\partial z_1}{\partial b_d} \\
		\vdots & \ddots & \vdots\\
		\frac{\partial z_d}{\partial b_1} & \hdots &
			\frac{\partial z_d}{\partial b_d} \end{pmatrix}
	\in \Mat_{d \times d}(A_P)
\]%
is invertible over $A_P$. That is, so that $\det(J)$ is a unit in $A_P$. Then choose $f$ so that $z_1, \dotsc, z_d \in A_f$ and $\det(J)$ is invertible in $A_f$, and set $\B_0 = \B \setminus \{b_1, \dotsc, b_d\}$. We claim that $\B_0 \cup \{z_1, \dotsc, z_d\}$ is an absolute $p$-basis of $A_f$.

In virtue of Corollary~\ref{crl:KimuraNiitsuma:open}, in order to prove the claim we just need to show that $\B_0 \cup \{z_1, \dotsc, z_d\}$ is an absolute differential basis of $A_f$. That is, that $\Kahler_{A_f} = \Kahler_A \otimes A_f$ is a free module with $\{d_A(b) \mid b \in \B_0 \} \cup \{d_A(z_1), \dotsc d_A(z_d) \}$ as a basis.

Since $\Kahler_{A_f} = \Kahler_A \otimes A_f$, one has that $\Kahler_{A_f}$ is free. On the other hand, observe that $J$ can also be regarded as a matrix over $A_f$. As $\det(J)$ is invertible in $A_f$, it follows that $J$ is also invertible as a matrix over $A_f$. Thus one can check that $\{d_A(b) \mid b \in \B_0 \} \cup \{d_A(z_1), \dotsc d_A(z_d) \}$ is a basis of $\Kahler_{A_f}$.
\end{proof}

\begin{lemma}  \label{lm:p-basis:quotient-open}
Let $A$ be a regular ring containing $\mathbb{F}_p$ that admits an absolute $p$-basis, and let $C = A / J$ be a regular quotient of $A$. Then, for each prime ideal $\p \subset C$, there exists an element $g \in C \setminus \q$ so that $C_g$ admits an absolute $p$-basis.
\end{lemma}

\begin{proof}
Let $P$ denote the preimage of $\p$ in $A$. Since $C$ is regular, one can find a regular system of parameters of $A_P$, say $z_1, \dotsc, z_d$, so that $JA_P = \langle z_1, \dotsc, z_r \rangle$, with $r = \height(JA_P)$. In this setting, according to Lemma~\ref{lm:p-basis:rsp-open}, it is possible to choose an element $f \in A \setminus P$ in such a way that $z_1, \dotsc, z_d \in A_f$, $JA_f = \langle z_1, \dotsc, z_r \rangle$, and $A_f$ has an absolute $p$-basis of the form $\B_0 \cup \{z_1, \dotsc, z_d\}$. Choose $g$ as the image of $f$ in $C$. Then we claim that the image of $\B_0 \cup \{z_{r+1}, \dotsc, z_d\}$ is an absolute $p$-basis of $C_g = A _f / J A_f$.

In order to prove the claim, we may assume without loss of generality that $A = A_f$, $J = J A_f$, and $C = C_g$. Recall that, by Proposition~\ref{prop:differential-basis}, $\Kahler_A$ is a free $A$-module with the set $d_A(\B_0) \cup \{d_A(z_1), \dotsc, d_A(z_d)\}$ as a basis. In virtue of
\cite[Theorem~25.2]{MatsumuraRingTheory},
there is a natural exact sequence
\[ \xymatrix {
	J / J^2 \ar[r]^-{\overline{d_A}} &
	\Kahler_{A} \otimes C \ar[r]^-{\upsilon} &
	\Kahler_{C} \ar[r] &
	0.
} \]%
Since $J$ is generated by $z_1, \dotsc, z_r$, the kernel of  $\upsilon$ is generated by the images of $d_A(z_1), \dotsc, d_A(z_r)$ in $\Kahler_A \otimes C$. Thus we see that $\Kahler_C$ is a free $C$-module with the image of the set $d_A(\B_0) \cup \{ d_A(z_{r+1}), \dotsc, d_A(z_d) \}$ as a basis. Hence the image of $\B_0 \cup \{z_{r+1}, \dotsc, z_d\}$ in $C_g$ is an absolute $p$-basis of $C_g$ by Corollary~\ref{crl:KimuraNiitsuma:open}
\end{proof}

\begin{remark}
As a consequence of Lemma~\ref{lm:p-basis:quotient-open} one has that, if $R$ is a regular local ring of characteristic $p > 0$ that admits an absolute $p$-basis and $R / J$ is a regular quotient of $R$, then $R/J$ also admits an absolute $p$-basis.
\end{remark}

\begin{proposition} \label{prop:p-basis:finite-type-extension}
Let $Z_0$ be a regular scheme of characteristic $p > 0$ that can be covered by open affine charts of the form $\Spec(A) \subset Z_0$ so that $A$ admits an absolute $p$-basis. Then any regular scheme $Z$ of finite type over $Z_0$ can be covered by open affine charts of the form $\Spec(C) \subset Z$ so that $C$ admits an absolute $p$-basis.
\end{proposition}

\begin{proof}
Observe that $Z$ can be covered by affine charts of the form
\[
	\Spec( A[x_1, \dotsc, x_n] / J ),
\]%
where $\Spec(A)$ is an affine chart of $Z_0$, and $J$ is an ideal of the polynomial ring $A[x_1, \dotsc, x_n]$. Then the result follows from Lemma~\ref{lm:p-basis:polynomial-extension} and Lemma~\ref{lm:p-basis:quotient-open}.
\end{proof}

\begin{theorem} \label{thm:variety-admits-p-basis}
Let $Z$ be a regular variety over a \textup{(}possibly non-perfect\textup{)} field $k$ of characteristic $p > 0$. For each $\xi \in Z$, there exists an open affine neighborhood of $Z$ at $\xi$, say $\Spec(C) \subset Z$, so that $C$ admits an absolute $p$-basis.
\end{theorem}

\begin{proof}
By assumption $Z$ is endowed with a finite type morphism $Z \to \Spec(k)$. Since every field admits an absolute $p$-basis, the claim follows from the previous proposition.
\end{proof}

\begin{remark}
It is an open problem to determine the class of schemes that admit $p$-bases or, more precisely, that can be covered by open affine charts of the form $\Spec(A)$, so that $A$ admits an absolute $p$-basis.

Note that a noetherian domain $A$ that has an absolute $p$-basis can be regarded as a free $A^p$-module (see Lemma~\ref{lm:p-basis-presentation}), and hence it is regular by a theorem of Kunz~\cite{Kunz69}.
In addition, a noetherian domain that has an absolute $p$-basis is excellent by \cite{KimuraNiitsuma79}. Thus the class of noetherian domains that admit an absolute $p$-basis is contained in the category of regular excellent domains.

On the other hand, there are examples of regular excellent domains of characteristic $p > 0$ which do not admit an absolute $p$-basis. Indeed, \cite[Example~3.8]{KimuraNiitsuma80} exhibits a complete regular local ring of dimension one which does not admit an absolute $p$-basis.
\end{remark}

\endgroup  %

\section{Jacobian criterion for regularity}
\label{sec:jacobian}

Let $A$ be a regular algebra of characteristic $p > 0$ that admits an absolute $p$-basis, and let $C = A / J$ be a quotient of $A$. In this section we give a Jacobian criterion for the regularity of $C$ at a prime $\q \in \Spec(C)$ (Proposition~\ref{prop:jacobian-criterion-p-bases}). This can be regarded as a generalization of Zariski's Jacobian criterion \cite[Theorem~11]{Zariski47}, which characterizes the regular points of an algebra that is given as the quotient of a polynomial ring (see Corollary~\ref{crl:jacobian-criterion-polynomials}).

\begin{proposition}  \label{prop:jacobian-criterion-p-bases}
Let $A$ be a regular ring containing $\mathbb{F}_p$ that admits an absolute $p$-basis, say $\B$. Let $J = \langle f_1, \dotsc, f_s \rangle$ be an ideal in $A$, and consider a prime ideal $P \subset A$ such that $J \subset P$. Set $C = A / J$ and $\p = P / J \subset C$. Then:
\begin{enumerate}[i\textup{)}]

\item There is a finite subset $\{b_1, \dotsc, b_m\} \subset \B$ so that $\frac{\partial f_1}{\partial b'} = \dotsb = \frac{\partial f_s}{\partial b'} = 0$\footnote{The notation $\frac{\partial}{\partial b'}$ was introduced in Corollary~\ref{crl:partial-b}.} for every $b' \in \B \setminus \{b_1, \dotsc, b_m\}$;

\item Assume that $\height(JA_P) = r$. Then $C_{\p}$ is regular if and only if the Jacobian matrix
\[
	\Jacobian \left (f; \tfrac{\partial}{\partial b} \right )
	:= \begin{pmatrix}
		\frac{\partial f_1}{\partial b_1} & \hdots &
			\frac{\partial f_1}{\partial b_m} \\
		\vdots & \ddots & \vdots\\
		\frac{\partial f_s}{\partial b_1} & \hdots &
			\frac{\partial f_s}{\partial b_m}
	\end{pmatrix}
\]%
has rank $r$ modulo $P$ \textup{(}that is, if and only if it contains an $r \times r$ minor whose determinant is non-zero modulo $P$\textup{)}.
\end{enumerate}
\end{proposition}

\begin{proof}
Since $p$-bases are compatible with localization (Lemma~\ref{lm:p-basis:localization}), we may assume without loss of generality that $A = A_P$ and $P = P A_P$. Note that, in this setting, $C = C_{\p}$ and $\p = \p C_{\p}$.

Let $K$ denote the residue field of $(A,P)$ and $(C,\p)$, and set $d = \dim(A) = \rank_K(P / P^2)$. As $A$ is catenary (see \cite[Proposition~5.6.4]{EGA-IV-2}), one has that
\[
	\dim(C) = \dim(A) - \height(J) = d-r .
\]%
Thus we see that $C_{\p}$ is regular if and only if $\rank_K(\p / \p^2) = d-r$.

In virtue of
\cite[Theorem~25.2]{MatsumuraRingTheory}
and Lemma~\ref{lm:2nd-exact-seq-local}, one can construct a natural commutative diagram as follows,
\[ \xymatrix{
	&
		J / J^2 \otimes K  \ar@{=}[r] \ar[d]^-{\lambda} &
		J / J^2 \otimes K  \ar[d]^-{\delta} &
		0  \ar[d] &
		\\
	0 \ar[r] &
		P / P^2  \ar[r]^-{\overline{d_A}} \ar[d]^-{\mu} &
		\Kahler_{A} \otimes K  \ar[r] \ar[d] &
		\Kahler_K  \ar[r] \ar[d] &
		0 \\
	0 \ar[r] &
		\p / \p^2  \ar[r]^-{\overline{d_C}} \ar[d] &
		\Kahler_{C} \otimes K  \ar[r] \ar[d] &
		\Kahler_K  \ar[r] \ar[d] &
		0, \\
	&
		0 &
		0 &
		0
} \]%
where all the rows and columns are exact. According to this diagram, $\delta = \overline{d_A} \circ \lambda$, and hence
\[
	\rank_K (\p / \p^2)
	= d - \rank_K \ker(\mu)
	= d - \rank_K \delta(J / J^2 \otimes K)
	.
\]%
Note that the linear map $\delta$ is given by the matrix $\Jacobian \bigl( f; \tfrac{\partial}{\partial b} \bigr)$ modulo~$P$. Then the rank of $\delta(J / J^2 \otimes K)$ coincides with that of the matrix $\Jacobian \bigl( f; \tfrac{\partial}{\partial b} \bigr)$ modulo~$P$. Thus we see that $C_{\p}$ is regular if and only if $\Jacobian \bigl( f; \tfrac{\partial}{\partial b} \bigr)$ has rank $r$ modulo~$P$.
\end{proof}

\begin{corollary}
Let $A$ be a regular ring of characteristic $p > 0$ which has an absolute $p$-basis, and let $C$ be an algebra of finite type over $A$. Then $\Reg(\Spec(C))$ is a \textup{(}possibly empty\textup{)} open subset of $\Spec(C)$.
\end{corollary}

\begin{proof}
Assume that $C = A[x_1, \dotsc, x_n] / J$. By Lemma~\ref{lm:p-basis:polynomial-extension}, if $\B$ is an absolute $p$-basis of $A$, then $\B \cup \{x_1\, \dotsc, x_n\}$ is an absolute $p$-basis of $A[x_1, \dotsc, x_n]$. Thus the result follows easily from the proposition.
\end{proof}

Consider a polynomial ring $k[x_1, \dotsc, x_n]$ over a field $k$ of characteristic $p > 0$, and let $\B_0$ be an absolute $p$-basis of $k$. In virtue of Corollary~\ref{crl:partial-b}, there is family of derivatives in $k$ canonically associated to $\B_0$, say $\left \{\frac{\partial}{\partial b} \mid b \in \B_0 \right \} \subset \Der(k)$. Each of these derivatives can be extended (non-canonically) to a derivative on $k[x_1, \dotsc, x_n]$ by setting $\frac{\partial x_1}{\partial b} = \dotsb = \frac{\partial x_n}{\partial b} = 0$. For the sake of readability, in the following proposition we shall abuse our notation and use the same symbol to refer to this extension of $\frac{\partial}{\partial b}$ to $k[x_1, \dotsc, x_n]$.

\begin{corollary}[{Zariski \cite[Theorem~11]{Zariski47}}]%
\label{crl:jacobian-criterion-polynomials}
Let $A = k[x_1, \dotsc, x_n]$ be a polynomial ring over a \textup{(}possibly non-perfect\textup{)} field $k$ of characteristic $p > 0$. Fix an absolute $p$-basis of $k$, say $\B_0$. Let $J = \langle f_1, \dotsc, f_s \rangle$ be an ideal in $A$, and consider a prime ideal $P \subset A$ such that $J \subset P$. Set $C = A / J$ and $\p = P / J \subset C$. Then:
\begin{enumerate}[i\textup{)}]

\item The set $\B_0 \cup \{x_1, \dotsc, x_n\}$ is an absolute $p$-basis of $A$, and there is a finite subset $\{b_1, \dotsc, b_m\} \subset \B_0$ so that $\frac{\partial f_1}{\partial b'} = \dotsb = \frac{\partial f_s}{\partial b'} = 0$ for every $b' \in \B \setminus \{b_1, \dotsc, b_m\}$;

\item Assume that $\height(JA_P) = r$. Then $C_{\p}$ is regular if and only if the Jacobian matrix
\[
	\Jacobian \left ( f;
		\tfrac{\partial}{\partial b},
		\tfrac{\partial}{\partial x} \right )
	:= \begin{pmatrix}
		\frac{\partial f_1}{\partial b_1} & \hdots &
			\frac{\partial f_1}{\partial b_m} &
			\frac{\partial f_1}{\partial x_1} & \hdots &
			\frac{\partial f_1}{\partial x_n} \\
		\vdots & \ddots & \vdots & \vdots & \ddots & \vdots \\
		\frac{\partial f_s}{\partial b_1} & \hdots &
			\frac{\partial f_s}{\partial b_m} &
			\frac{\partial f_s}{\partial x_1} & \hdots &
			\frac{\partial f_s}{\partial x_n}
	\end{pmatrix}
\]%
has rank $r$ modulo $P$ \textup{(}that is, if and only if it contains an $r \times r$ minor whose determinant is non-zero modulo $P$\textup{)}.
\end{enumerate}
\end{corollary}

\begin{proof}
By Lemma~\ref{lm:p-basis:localization}, the set $\B = \B_0 \cup \{x_1, \dotsc, x_n\}$ is an absolute $p$-basis of $A$. Thus the result follows immediately from Proposition~\ref{prop:jacobian-criterion-p-bases}.
\end{proof}

\begin{example} \label{ex:jacobian-criterion}
Consider the field $\mathbb{F}_p(v)$, where $v$ represents a transcendental element, the polynomial $f = x^p + v y^p \in k[x,y]$, and the variety
\[
	H = \Spec( k[x,y] / \langle f \rangle )
	\subset \mathbb{A}^2_k .
\]%
Observe that $H$ has codimension $1$ in $\mathbb{A}^2_k$, and that a $p$-basis of $k[x,y]$ is given by $\{ v, x, y \}$ (see Lemma~\ref{lm:p-basis:polynomial-extension}). Thus, according to Corollary~\ref{crl:jacobian-criterion-polynomials}, in order to compute the singular locus of $H$, we shall calculate the Jacobian matrix
\[
	\Jacobian \left ( f;
		\tfrac{\partial}{\partial v},
		\tfrac{\partial}{\partial x},
		\tfrac{\partial}{\partial y} \right )
	= \begin{pmatrix}
		\frac{\partial f}{\partial v} &
			\frac{\partial f}{\partial x} &
			\frac{\partial f}{\partial y}
		\end{pmatrix}
	= \begin{pmatrix}
		y^p &
			0 &
			0
		\end{pmatrix} .
\]%
Clearly, this matrix has rank $1$ everywhere in $H$ except at the point of the origin. Hence the singular locus of $H$ consists only of the point at the origin.
\end{example}

\section{Absolute \texorpdfstring{$p$}{p}-bases and differential operators of higher order}
\label{sec:differential-operators}

Let $A$ be a regular ring of characteristic $p > 0$. In this section we show that, given an absolute $p$-basis of $A$, say $\B$, there is a natural family of differential operators on $A$ associated to $\B$ which behaves mostly like the family of differential operators with respect to the variables on a polynomial ring. The following lines are intended to fix the setting and notation for Proposition~\ref{prop:p-basis:diffs}.

\begin{definition}
Consider a ring $k$, a $k$-algebra $A$, and an $A$-module $M$.
A map $\Delta : A \to M$ is said to be a \emph{differential operator of order zero from $A$ to $M$ over $k$} if it is $A$-linear. For $n > 0$, a map $\Delta : A \to M$ is said to be a \emph{differential operator of order $n$ over $k$} if it is $k$-linear and, for every $a \in A$,
\begin{align*}
	[a,\Delta] : A &\longrightarrow M \\
	f &\longmapsto a \Delta(f) - \Delta(af)
\end{align*}
is a differential operator of order $n-1$ over $k$.
\end{definition}

\begin{lemma}[{\cite[Ch.~III, Lemma~1.2.2]{GiraudEtudeLocale}}]
\label{lm:p-basis:pe-linear}
Let $k$ be a ring of characteristic $p > 0$, let $A$ be an arbitrary $k$-algebra, and let $M$ be an $A$-module. If $\Delta : A \to M$ is a differential operator of order $n$ over $k$, then $\Delta$ is $k[A^{p^e}]$-linear for every $p^e > n$.
\end{lemma}

\subsection*{Multi-index notation} Let $\Lambda$ be a possibly infinite set of indices. We shall denote by $\mathbb{N}^{\oplus\Lambda}$ the set of $\Lambda$-tuples with almost all entries equal to zero (i.e., except for a finite number of them). Given a tuple $\beta = (\beta_\lambda) \in \mathbb{N}^{\oplus\Lambda}$, we define the order of $\beta$ as $\lvert \beta \rvert := \sum_{\lambda \in \Lambda} \beta_\lambda$. Given two tuples $\alpha = (\alpha_\lambda)$ and $\beta = (\beta_\lambda)$, we will say that $\alpha \leq \beta$ if $\alpha_\lambda \leq \beta_\lambda$ for all $\lambda \in \Lambda$. Similarly, we will say that $\alpha < \beta$ if $\alpha \leq \beta$ and $\alpha \neq \beta$. For two tuples $\alpha, \beta \in \mathbb{N}^{\oplus\Lambda}$, we shall define the binomial coefficient
\[
	\binom{\alpha}{\beta}
	:= \prod_{\lambda \in \Lambda}
		\binom{\alpha_\lambda}{\beta_\lambda} \in \mathbb{N},
\]%
taking $\binom{\alpha_\lambda}{\beta_\lambda} = 0$ whenever $\alpha_\lambda < \beta_\lambda$. In addition, given a ring $R$ and a $\Lambda$-indexed subset of $R$, say $T = \{t_\lambda \mid \lambda \in \Lambda\} \subset R$, for each $\beta \in \mathbb{N}^{\oplus\Lambda}$ we shall define the monomial
\[
	T^\beta
	:= \prod_{\lambda \in \Lambda} t_\lambda^{\beta_\lambda} \in A.
\]%

\subsection*{Differential operators on a polynomial ring} Given a ring $A$ and possibly infinite set of variables $X = \{x_\lambda \mid \lambda \in \Lambda\}$, consider the polynomial ring $A[X]$. Let $T = \{t_\lambda \mid \lambda \in \Lambda\}$ be a different set of variables, and set $X+T := \{ x_\lambda + t_\lambda \mid \lambda \in \Lambda\}$. For a polynomial $F(X) \in A[X]$, let us consider the following expansion of $F(X+T) \in A[X,T]$:
\[
	F(X+T)
	= \sum_{\alpha \in \mathbb{N}^{\oplus \Lambda}}
		F_\alpha(X) \cdot T^\alpha,
	\qquad F_\alpha(X) \in A[X].
\]%
Then, for each $\beta \in \mathbb{N}^{\oplus \Lambda}$, we shall define the Taylor operator $\Tay^\beta : A[X] \to A[X]$ by $\Tay^\beta\bigl(F(X)\bigr) = F_\beta(X)$. Note that $\Tay^\beta$ is $A$-linear for every $\beta$. In fact, it can be proved that $\Tay^\beta$ is a differential operator of order $\lvert \beta \rvert$ over $A$ for every $\beta \in \mathbb{N}^{\oplus \B}$ (see \cite[Ch.~III, Example~1.2.4]{GiraudEtudeLocale}). In addition, one can check that, for $\alpha \in \mathbb{N}^{\oplus \Lambda}$,
\[
	\Tay^\beta(X^\alpha)
	= \binom{\alpha}{\beta} X^{\alpha - \beta}.
\]%

\begin{proposition}  \label{prop:p-basis:diffs}
Let $A$ be a reduced ring containing $\mathbb{F}_p$ that has an absolute $p$-basis, say $\B$. Then, for each $\beta \in \mathbb{N}^{\oplus \B}$, there exists a differential operator of order $\abs{\beta}$ over $\mathbb{F}_p$, say $D^{[\B;\beta]}_{A} : A \to A$, such that
\begin{equation}  \label{eq:p-basis:prop-diffs-binom-condition}
	D^{[\B;\beta]}_{A} (\B^\alpha) = \binom{\alpha}{\beta} \B^{\alpha - \beta}
\end{equation}%
for all $\alpha \in \mathbb{N}^{\oplus \B}$.
\end{proposition}

\begin{proof}
\newcommand{\A}{A^{p^e}[T]}
Fix $\beta \in \mathbb{N}^{\oplus \B}$, choose $e > 0$ so that $p^e > \abs{\beta}$, and let $T = \{t_b \mid b \in \B \}$ be a set of variables. By definition of $p$-basis, we have that
\[
	A \simeq A^p [T]
		/ \bigl\langle t_b^p - b^p \mid b \in \B \bigr\rangle.
\]%
Since $A$ is reduced, $A \simeq A^p$ via the Frobenius isomorphism. Thus, by induction on $e$, we get
\[
	A \simeq A^{p^e} [T]
		/ \bigl\langle t_b^{p^e} - b^{p^e} \mid b \in \B \bigr\rangle.
\]%
Set $J = \bigl\langle t_b^{p^e} - b^{p^e} \mid b \in \B \bigr\rangle \subset \A$, in such a way that $A \simeq \A / J$. Regarding $\A$ as a polynomial ring over $A^{p^e}$, and according to the discussion preceding this proposition, there is an $A^{p^e}$-linear differential operator $\Tay^\beta : \A \to \A$ such that
\[
	\Tay^\beta(T^\alpha) = \binom{\alpha}{\beta} T^{\alpha - \beta}
\]%
for $\alpha \in \mathbb{N}^{\oplus \B}$. We will show that $\Tay^\beta$ induces a natural differential operator $D^{[\B;\beta]}_{A} : A \to A$ of order $\abs{\beta}$ over $\mathbb{F}_p$ via the quotient map $\A \to A \simeq \A / J$.

By composing $\Tay^\beta$ with the map $\A \to A$, we obtain a differential operator $\overline{\Tay^\beta} : \A \to A$. Note that $\overline{\Tay^\beta}(T^\alpha) = \binom{\alpha}{\beta} \B^{\alpha - \beta}$ for $\alpha \in \mathbb{N}^{\oplus \B}$. In order to see that $\overline{\Tay^\beta}$ induces a differential operator $D^{[\B;\beta]}_{A} : A \to A$, we shall verify that it annihilates the ideal $J \subset \A$.

Recall that $\Tay^\beta : \A \to \A$ is a differential operator of order $\abs{\beta} < p^e$ over $A^{p^e}$. Therefore, by Lemma~\ref{lm:p-basis:pe-linear}, $\Tay^\beta$ is linear over the subring
\[
	A^{p^e} \left[ t_b^{p^e} \mid b \in \B \right]
	\subset \A
	.
\]%
Thus we see that, for each element $b \in \B$ and each polynomial $f \in \A$,
\[
	\Tay^\beta \left ( \bigl( t_b^{p^e} - b^{p^e} \bigr) \cdot f \right )
	= \bigl( t_b^{p^e} - b^{p^e} \bigr) \cdot \Tay^\beta(f)
	\in J.
\]%
This implies that $\overline{\Tay^\beta}$ annihilates $J$, and hence it induces a differential operator on $A$, say $D^{[\B;\beta]}_{A} : A \to A$, as required.
\end{proof}

\begin{remark}
At first glance, it may seem that the definition $D^{[\B;\beta]}_A$ depends on the choice of the integer $e \gg 0$. However, as follows from Corollary~\ref{crl:p-basis:diff-smooth} below, there is a unique differential operator on $A$ of order $\abs{\beta}$ over $\mathbb{F}_p$ satisfying condition \eqref{eq:p-basis:prop-diffs-binom-condition} of Proposition~\ref{prop:p-basis:diffs}. Thus $D^{[\B;\beta]}_A$ is well-defined, in the sense that it does not depend on the choice of the integer $e$.
\end{remark}

\begin{corollary}  \label{crl:p-basis:diff-smooth}
Let $A$ be a reduced ring containing $\mathbb{F}_p$ that admits an absolute $p$-basis, say $\B$. Then $A$ is differentially smooth over $\mathbb{F}_p$ \textup{(}in the sense of Grothendieck \cite[\S16.10]{EGA-IV-4}\textup{)}. Moreover, the family of differential operators $D^{[\B;\beta]}$ \textup{(}as defined in Proposition~\ref{prop:p-basis:diffs}\textup{)} has the following properties:
\begin{enumerate} [i\textup{)}]
\item For all $\beta, \beta' \in \mathbb{N}^{\oplus \B}$,
\[
	D_A^{[\B;\beta]} \circ D_A^{[\B;\beta']}
	= D_A^{[\B;\beta']} \circ D_A^{[\B;\beta]}
	= \frac{(\beta + \beta')!}{\beta! \beta'!} \cdot D_A^{[\B;\beta + \beta']} .
\]%
\item For any element $f \in A$, there exists a finite number of indices $\beta \in \mathbb{N}^{\oplus \B}$ so that $D_A^{[\B;\beta]} (f) \neq 0$.
\item Every differential operator from $A$ to an $A$-module $M$ of order $n$ \textup{(}over $\mathbb{F}_p$\textup{)}, say $\Delta : A \to M$, has a decomposition of the form
\[
	\Delta
	= \sum_{\abs{\beta} \leq n} c_\beta \cdot D_A^{[\B;\beta]} ,
\]%
with
\[
	c_\beta
	= \sum_{\gamma \leq \beta}
		\binom{\beta}{\gamma}
		(-1)^{\abs{\gamma}}
		\cdot \B^\gamma
		\cdot \Delta(\B^{\beta - \gamma})
	\in M
\]%
for $\abs{\beta} \leq n$ \textup{(}note that, although $\beta$ runs over a possibly infinite set, $\Delta(f)$ is a finite sum for each $f \in A$ by the previous condition\textup{)}.
\end{enumerate}%
\end{corollary}

\begin{proof}
By Proposition~\ref{prop:differential-basis}, $\Kahler_A$ is a free $A$-module with $\{d_A(b) \mid b \in \B \}$ as basis. Thus the differential smoothness and properties i), ii), and iii) follow from \cite[Theorem~16.11.2]{EGA-IV-4}.
\end{proof}

\section{Differential saturation and order of ideals}

\label{sec:differential-saturation-and-order}

In this section we intend to characterize the set of points of a variety where a given ideal has order greater than or equal to certain integer $N$. To this end, we shall make use of the differential operators introduced in the previous section.

\begin{definition}
Let $A$ be a regular ring and fix a prime ideal $\q \subset A$. The \emph{order of an element $f \in A$ at $\q$} is defined by
\[
	\ord_{\q}(f)
	= \max \left \{ n \in \mathbb{N}
		\mid f \in \q^n A_{\q} \right \}.
\]%
Similarly, the \emph{order of an ideal $I \subset A$ at $\q$} is defined by
\[
	\ord_{\q}(I)
	= \max \left \{ n \in \mathbb{N}
		\mid I A_{\q} \subset \q^n A_{\q} \right \} .
\]%
Given a regular variety $Z$ and a coherent ideal sheaf $\II \subset \OO_Z$, we define the \emph{order of $\II$ at a point $\xi \in Z$} as
\[
	\ord_\xi(\II)
	= \max \left \{ n \in \mathbb{N}
		\mid \II_\xi \subset \M_\xi^n \right \}.
\]%
\end{definition}

Let $A$ be a regular ring and, for an integer $s \geq 0$, let $\Diff^s(A)$ denote the module of absolute differential operators of order at most $s$ from $A$ to itself. Fix an element $f \in A$ and a prime ideal $\q \subset A$. It is known that, if $\ord_{\q}(f) \geq N$, then $\Delta(f) \in \q$ for every $\Delta \in \Diff^{N-1}(A)$ (cf. \cite[Ch.~III, Lemma~1.2.3]{GiraudEtudeLocale}). Moreover, in the case that $A$ is a polynomial ring over field $k$, say $A = k[x_1, \dotsc, x_r]$, the previous implication turns out to be an equivalence (see \cite[Ch.~III, Lemma~1.2.7]{GiraudEtudeLocale}). That is, when $A$ is a polynomial ring, $\ord_{\q}(f) \geq N$ if and only if $\Delta(f) \in \q$ for every $\Delta \in \Diff^{N-1}(A)$. This result has a natural generalization to the case in which $A$ admits an absolute $p$-basis, as shown in the following proposition.

\begin{proposition}  \label{prop:generalization-giraud-p}
Let $A$ be a regular ring over a field of characteristic $p > 0$ which has an absolute $p$-basis. Then, for any element $f \in A$ and any prime ideal $\q \subset A$, the following conditions are equivalent:
\begin{enumerate}[i\textup{)}]
\item $\ord_{\q}(f) \geq N$.
\item $\Delta(f) \in \q$ for every $\Delta \in \Diff^{N-1}(A)$.
\end{enumerate}%
\end{proposition}

\begin{proof}
i) $\Rightarrow$ ii) follows from \cite[Ch.~III, Lemma~1.2.3]{GiraudEtudeLocale}. For the converse we proceed by contradiction. Namely, assume that ii) holds and that $\ord_{\q}(f) = n$ with $n < N$.

Fix a regular system of parameters of $A_\q$, say $z_1, \dotsc, z_d$, and let $\widehat{A_{\q}}$ denote the completion of $A_{\q}$ with respect to its maximal ideal. Recall that, by Cohen's structure theorem,
$\widehat{A_{\q}} \simeq k(\q)[[z_1, \dotsc, z_d]]$, where $k(\q)$ represents the residue field of $A_{\q}$. Since $\widehat{A_{\q}}$ is faithfully flat over $A_\q$, the order of $f$ in $\widehat{A_{\q}}$ coincides with its order in $A_\q$. Hence $f$ has order $n$ in $\widehat{A_{\q}}$, and therefore there should be a differential operator in $\widehat{A_{\q}}$ of order at most $n$ over $k(\q)$, say $\Delta^* : \widehat{A_{\q}} \to \widehat{A_{\q}}$, so that $\Delta^*(f)$ is a unit in $\widehat{A_{\q}}$. That is, so that $\Delta^*(f) \notin \q \widehat{A_{\q}}$.

Next, fix an absolute $p$-basis of $A$, say $\B$. Recall that, in virtue of Proposition~\ref{prop:p-basis:diffs}, there is family of differential operators on $A$ naturally associated to $\B$, say $\left \{ D_A^{[\B;\beta]} \mid \beta \in \mathbb{N}^{\oplus \B} \right( \}$, where $D_A^{[\B;\beta]} \in \Diff^{\abs{\beta}}(A)$ for each $\beta \in \mathbb{N}^{\oplus \B}$. Let $\iota : A \to \widehat{A_{\q}}$ denote the natural homomorphism from $A$ to $\widehat{A_\q}$. By composition, $\Delta^* \circ \iota$ is an absolute differential operator of order $n$ from $A$ to $\widehat{A_{\q}}$. Thus, by Corollary~\ref{crl:p-basis:diff-smooth}~ii) and iii), there exists a finite collection of indices, say $\beta_1, \dotsc, \beta_r \in \mathbb{N}^{\oplus \B}$, with $\abs{\beta_i} \leq n$, such that
\[
	(\Delta^* \circ \iota)(f)
	= \sum_{i=1}^r \Delta(\B^{\beta_i}) D_A^{[\B;\beta_i]}(f).
\]%
Since $\Delta^*(f) \notin \q \widehat{A_{\q}}$, it follows that $D_A^{[\B;\beta_i]}(f) \notin \q$ for some $i \in \{1, \dotsc, r\}$, which contradicts ii).
\end{proof}

\subsection*{Differential saturation of ideals}

Consider a regular ring $A$ and an ideal $I \subset A$, and fix an integer $n \geq 0$. We shall define the \emph{saturation of $I$} by differential operators of order at most $n$ by
\[
	\Diff^n (A)(I)
	= \left \langle \Delta(f)
		\mid f \in I,  \Delta \in \Diff^n(A) \right \rangle.
\]%
Assume that the ring $A$ has an absolute $p$-basis. In this case, attending to Proposition~\ref{prop:generalization-giraud-p}, it is easy to see that $I$ has order at least $N$ at a prime ideal $\q \subset A$ if and only if $\Diff^{N-1}(I) \subset \q$. Thus, when $A$ has an absolute $p$-basis,
\[
	\bigl\{ \q \in \Spec(A) \mid \ord_{\q}(I) \geq N \bigr\}
	= \VV \bigl( \Diff^{N-1}(A)(I) \bigr).
\]%
This gives a characterization of the set of points of $\Spec(A)$ where the ideal $I$ has order at least $N$. The purpose of the rest of this section is to find a similar characterization for the set of points of a given variety $Z$ where a coherent ideal $\II \subset \OO_Z$ has certain order. To this end, we shall first check that the ideal $\Diff^n (A)(I)$ introduced above is compatible with localization.

\begin{lemma}  \label{lm:diff-sat-localization}
Let $A$ be regular ring of characteristic $p > 0$ which has an absolute $p$-basis, say $\B$, and let $I \subset A$ be an ideal. Fix an integer $n \geq 0$. Then, for any multiplicative subset $S \subset A$,
\begin{equation} \label{eq:lm:diff-sat-localization}
	\Diff^n (S^{-1}A) (S^{-1}I)
	= \Diff^n (A)(I) \otimes_A S^{-1}A .
\end{equation}%
\end{lemma}

\begin{proof}
Fix a collection of generators of $I$, say $I = \langle f_1, \dotsc, f_r \rangle$. For each $f_i$, consider the ideal
\[
	\Diff^n(A)(f_i)
	= \left \langle \Delta(f_i)
		\mid \Delta \in \Diff^n(A) \right \rangle.
\]%
Note that, with this notation,
\[
	\Diff^n (A)(I)
	= \Diff^n (A)(f_1) + \dotsb + \Diff^n (A)(f_r).
\]%
Moreover, by Corollary~\ref{crl:p-basis:diff-smooth}~iii), one has that
\[
	\Diff^n (A)(f_i)
	= \left \langle D^{[\B;\beta]}_A (f_i)
		\mid \, \abs{\beta} \leq n \right \rangle ,
\]%
and hence
\[
	\Diff^n (A)(I)
	= \left \langle D^{[\B;\beta]}_A (f_i)
		\mid i \in \{1, \dotsc, r\}; \,
		\abs{\beta} \leq n \right \rangle.
\]%
Paralleling the previous argument, and taking into account that $\B$ can also be regarded as $p$-basis of $S^{-1}A$ (Lemma~\ref{lm:p-basis:localization}), one can see that
\[
	\Diff^n (S^{-1}A) (S^{-1}I)
	= \left \langle D^{[\B;\beta]}_{S^{-1}A} (f_i)
		\mid i \in \{1, \dotsc, r\}; \,
		\abs{\beta} \leq n \right \rangle.
\]%
In this way, as the image of $D^{[\B;\beta]}_A (f_i)$ in $S^{-1}A$ coincides with $D^{[\B;\beta]}_{S^{-1}A} (f_i)$, one readily checks that \eqref{eq:lm:diff-sat-localization} holds.
\end{proof}

Next, given a regular variety $Z$ defined over a field positive characteristic $k$ and a coherent ideal $\II \subset \OO_Z$, we would like to define some sort of saturation of $\II$ by differential operators on $\OO_Z$. Note that, in general, the sheaf of absolute differential operators of order at most $n$ on $Z$, say $\DDiff^n(\OO_Z)$, is not finitely generated or quasi-coherent (namely, this occurs whenever $[k:k^p] = \infty$). Thus, it is not clear a priori whether one can construct a coherent differential saturation of $\II$.

\begin{proposition}  \label{prop:coherent-diff-I}
Let $Z$ be a regular variety over a \textup{(}possibly non-perfect\textup{)} field $k$ of characteristic $p > 0$, and let $\II$ be a coherent sheaf of ideals over $Z$. Then, for each integer $n \geq 0$, there exists a unique coherent sheaf of ideals over $Z$, which we shall denote by $\DDiff^n(\OO_Z)(\II)$, with the following properties:
\begin{enumerate}[i\textup{)}]
\item For every open subset $W \subset Z$, every element $f \in \Gamma(\II,W)$, and every differential operator $\Delta \in \Diff^n \bigl( \Gamma(\OO_Z,W) \bigr)$,
\[
	\Delta(f)
	\in \Gamma \bigl( \DDiff^n(\OO_Z)(\II), W \bigr);
\]%
\item For every $\xi \in Z$,
\[
	\bigl( \DDiff^n(\OO_Z)(\II) \bigr)_\xi
	= \Diff^n (\OO_{Z,\xi})(\II_\xi).
\]%
\end{enumerate}
\end{proposition}

\begin{proof}
We shall start by constructing the sheaf $\DDiff^n(\OO_Z)(\II)$. After that we will check that it satisfies conditions i), ii), and iii).

Let $\mathfrak{A}$ denote the family of open affine subschemes of $Z$ of the form $U = \Spec(A)$ such that $A$ admits an absolute $p$-basis. Recall that, by Theorem~\ref{thm:variety-admits-p-basis}, $Z$ is covered by these charts. For each chart $U \in \mathfrak{A}$, say $U = \Spec(A)$, consider the sheaf
\[
	\DD_U := \bigl[ \Diff^n(A) (\Gamma(\II,U)) \bigr]^{\sim}
\]%
defined over $U$. By Lemma~\ref{lm:diff-sat-localization},
\begin{equation} \label{eq:thm:coherent-diff-I:local}
	(\DD_U)_\xi
	= \Diff^n (\OO_{Z,\xi})(\II_\xi)
\end{equation}%
for every $\xi \in U$. Thus we see that the collection of sheafs $\{\DD_U\}_{U \in \mathfrak{A}}$ can be patched in order to obtain a coherent sheaf of ideals over $Z$. This will be the sheaf $\DDiff^n(\OO_Z)(\II)$.

To check property i), fix an open subset $W \subset Z$, an element $f \in \Gamma(\II,W)$, and a differential operator $\Delta \in \Diff^n \bigl( \Gamma(\OO_Z,W) \bigr)$. Choose a finite collection of charts, say $U_1, \dotsc, U_m \in \mathfrak{A}$, with $U_i = \Spec(A_i)$, so that $W = \bigcup_{i=1}^m U_i$. Note that, for each $i$, the element $f$ maps to a section $f_i \in A_i$, and $\Delta$ maps to a differential operator $\Delta_i \in \Diff^n(A_i)$, in such a way that $\Delta(f)$ maps to the section $\Delta_i(f_i) \in A_i$. By construction,
\[
	\Delta_i(f_i)
	\in \Gamma \bigl( \DDiff^n(\OO_Z)(\II), U_i \bigr)
	= \Diff^n(A_i) (\Gamma(\II,U_i)),
\]%
and hence the section $\Delta(f)$ belongs to $\Gamma \bigl( \DDiff^n(\OO_Z)(\II), W \bigr)$. This proves i). Property ii) follows from the construction and condition \eqref{eq:thm:coherent-diff-I:local}. %
\end{proof}

\begin{theorem}  \label{thm:order-ideal-sheaf}
Let $Z$ be a regular variety over a \textup{(}possibly non-perfect\textup{)} field $k$ of characteristic $p > 0$, and let $\II$ be a coherent sheaf of ideals over $Z$. Then, for every integer $N \geq 1$,
\[
	\bigl\{ \xi \in Z \mid \ord_\xi(\II) \geq N \bigr\}
	= \VV \bigl( \DDiff^{N-1}(\OO_Z)(\II) \bigr)
	\subset Z,
\]%
where $\DDiff^{N-1}(\OO_Z)(\II) \subset \OO_Z$ denotes the coherent ideal sheaf constructed in Proposition~\ref{prop:coherent-diff-I}.
\end{theorem}

\begin{proof}
This is a consequence of Proposition~\ref{prop:coherent-diff-I}~ii) and Proposition~\ref{prop:generalization-giraud-p}.
\end{proof}

\bibliographystyle{abbrv} 

\end{document}